\documentclass[12pt,reqno]{amsart}

\usepackage[T1]{fontenc}
\usepackage[utf8]{inputenc} 

\usepackage{amsmath,amssymb,amsthm}
\usepackage{mathtools}
\usepackage{enumitem}
\usepackage{array}
\usepackage{multicol}
\usepackage{verbatim}
\usepackage{tikz}
\usepackage{float}
\usetikzlibrary{shapes,backgrounds,calc}
\usepackage{graphics}
\usepackage{accents}
\usepackage{hyperref}
\usepackage{setspace}
\usepackage{pstricks}
\usepackage{geometry}

\hypersetup{
  colorlinks=true,
  linkcolor=blue,
  citecolor=blue,
  urlcolor=blue
}

\theoremstyle{plain}
\newtheorem{cor}{Corollary}
\newtheorem{lem}{Lemma}

\newtheorem{theorem}{Theorem}
\newtheorem*{thm}{Theorem}

\theoremstyle{definition}
\newtheorem{defn}{Definition}

\newtheorem{example}{Example}
\newtheorem{rem}{Remark}
\newtheoremstyle{block}
  {\topsep}   
  {\topsep}   
  {\normalfont}  
  {}          
  {\bfseries} 
  {.}         
  {\newline}  
  {}          

\theoremstyle{block}

\theoremstyle{remark}

\numberwithin{equation}{section}

\newcommand{\R}{{\mathbb {R}}}
\begin{document}

\title{The P-Vertex Problem for Graphs with Perfect Matchings}

   \author{G. Arunkumar}
	
	\address{Indian Institute of Technology Madras, Chennai, India.}
	\email{garunkumar@iitm.ac.in}

	\author{U. S. Jerisha$^{\ast}$}
   \thanks{}
	\address{Indian Institute of Technology Madras, Chennai, India.}
	\email{ma24d001@smail.iitm.ac.in}


\begin{abstract}
Sharma and Panda recently proved that every bipartite graph with a perfect matching has property ($P$); that is, it admits a non-singular real symmetric matrix with support graph $G$ for which every vertex is a $P$-vertex. In this paper, we extend their result from bipartite graphs to arbitrary graphs. To this end, we introduce the notion of a $P$-vertex covering and define the $P$-vertex covering number $p(G)$ as the minimum number of non-singular matrices in $S(G)$ needed so that every vertex of $G$ is a $P$-vertex of at least one of them. Given a maximal matching of $G$, we partition the vertex set into the vertices saturated by the matching and the remaining vertices, which necessarily form an independent set. We then construct separate matrices covering these two classes of vertices. We use the Implicit Function Theorem as a perturbation tool to establish the desired result.

\end{abstract}

\maketitle

\medskip


\section{Introduction}

The relationship between the combinatorial structure of a graph and the
spectral properties of associated matrices is a central topic in spectral
graph theory. One aspect of this relationship concerns the effect of
vertex deletion on eigenvalue multiplicities.

Let \(A=(a_{ij})\in\mathbb{R}^{n\times n}\) be a real symmetric matrix.
The \emph{support graph} of \(A\), denoted by \(G(A)\), is the graph with
vertex set \([n]=\{1,\ldots,n\}\) and edge set
\[
E(G(A))
=
\bigl\{\{i,j\}:i\neq j \text{ and }a_{ij}\neq0\bigr\}.
\]
Thus, the diagonal entries of \(A\) do not affect its support graph. For
a graph \(G\) on \(n\) vertices, define
\[
S(G)
=
\bigl\{
A\in\mathbb{R}^{n\times n}:
A^{T}=A \text{ and }G(A)=G
\bigr\}.
\]

For \(i\in[n]\), let \(A(i)\) denote the principal submatrix obtained
from \(A\) by deleting its \(i^{\text{th}}\) row and column. If \(\lambda\) is an
eigenvalue of \(A\), let \(m_A(\lambda)\) denote its algebraic
multiplicity. By the Cauchy interlacing theorem,
\[
m_{A(i)}(\lambda)-m_A(\lambda)\in\{-1,0,1\}.
\]
A vertex \(i\) is called a \emph{\(P\)-vertex of \(A\)} if
\[
m_{A(i)}(0)-m_A(0)=1.
\]
Equivalently, deleting the row and column corresponding to \(i\)
increases the nullity by exactly one.

The study of such vertices originates in the work of Parter
\cite{Parter}, who investigated changes in eigenvalue multiplicities
under vertex deletion for symmetric matrices whose underlying graphs
are trees. Wiener \cite{wiener} subsequently extended this line of
research to sign-symmetric acyclic matrices. Related ideas were also
studied by Godsil in connection with matching polynomials
\cite{godsil}. Johnson and Sutton \cite{hermitian} developed a general
framework for classifying vertices according to the change in
eigenvalue multiplicity caused by their deletion. The notions of
\(P\)-vertices and \(P\)-sets were later studied systematically by Kim
and Shader \cite{kim1,kim2}.

Much of the subsequent work focused on acyclic matrices and the
maximum number of \(P\)-vertices that such matrices can possess.
Several extremal and structural questions for trees were investigated
in \cite{fons1,fons2,fons5,du2}. The study was later extended beyond
trees. Andeli\'c, da Fonseca, and Mamede \cite{fons1} showed, in
particular, that every cycle has the full \(P\)-vertex property.
Howlader, Raickwade, and Sivakumar \cite{Howlader} subsequently studied
the corresponding problem for unicyclic graphs.

A graph \(G\) on \(n\) vertices is said to have \emph{property
\((P)\)} if there exists a non-singular matrix \(A\in S(G)\) for which
every vertex of \(G\) is a \(P\)-vertex. Since \(A\) is non-singular,
the cofactor identity gives
\[
(A^{-1})_{ii}
=
\frac{\det A(i)}{\det A}.
\]
Consequently, every vertex of \(G\) is a \(P\)-vertex of \(A\) if and
only if
\(
\operatorname{diag}(A^{-1})=\mathbf 0.
\)
Thus, property \((P)\) may equivalently be formulated as the existence
of a non-singular matrix in \(S(G)\) whose inverse is hollow.

Recently, Sharma and Panda \cite{sharma} developed a graph-theoretic
framework for studying property \((P)\), with particular emphasis on
graphs admitting perfect matchings. One of their principal results
states that every connected bipartite graph with a perfect matching
has property \((P)\). Since both property \((P)\) and the existence of
a perfect matching are determined componentwise, their result extends
immediately to arbitrary bipartite graphs. They further showed that a
tree has property \((P)\) if and only if it has a perfect matching,
which is necessarily unique, and established an analogous
characterization for bipartite unicyclic graphs. They also investigated
certain non-bipartite graphs with unique perfect matchings and the
behavior of property \((P)\) under graph operations.

The requirement in property \((P)\) that a single matrix make every
vertex a \(P\)-vertex is quite restrictive. This motivates a covering
version of the problem. We say that a graph \(G\) admits a
\emph{\(P\)-vertex covering} if there exists a finite collection
\[
\mathcal A=\{A_1,\ldots,A_k\}\subseteq S(G)
\]
of non-singular matrices such that every vertex of \(G\) is a
\(P\)-vertex of at least one matrix in \(\mathcal A\). The minimum
cardinality of such a collection is called the
\emph{\(P\)-vertex covering number} of \(G\) and is denoted by \(p(G)\).
If no such collection exists, we define
\(
p(G)=\infty.
\)

\noindent With this terminology,
\[
p(G)=1
\quad\Longleftrightarrow\quad
G\text{ has property }(P).
\]

The main objective of this paper is to prove the following theorem in which we determine which graphs admit a
\(P\)-vertex covering and to obtain a universal bound for their
covering numbers. 

\begin{thm}
 Every graph without isolated vertices either has property
\((P)\) or has \(P\)-vertex covering number exactly \(2\).  
\end{thm}


\textit{Proposition~7.8 of \cite{fons1} claims that the central vertex of a 2-generalized star \(T\) cannot be a \(P\)-vertex for any non-singular matrix in \(S(T)\). We have observed that this result is false. Indeed, the graph considered
there is a tree with no perfect matching. Since a tree possesses
property~\((P)\) if and only if it has a perfect matching, no
non-singular matrix in \(S(G)\) can make every vertex of this graph a
\(P\)-vertex. The error in the proof lies in the implicit assumption
that the chosen non-singular matrix makes every vertex a \(P\)-vertex. Hence this result does not violate our preceding result.}

The proof of our theorem is based on a maximal matching of the graph. Let \(M\) be a maximal matching, let \(W\) be the set of vertices saturated by \(M\) and let \(U=V(G)\setminus W.\)
The maximality of \(M\) implies that \(U\) is independent. We first
construct a non-singular matrix in \(S(G)\) for which every vertex in
\(W\) is a \(P\)-vertex. The initial matrix is block diagonal, with one
block
\(
\begin{bmatrix}
0&1\\
1&0
\end{bmatrix}
\)
for each edge of \(M\) and a \(1\times1\) block \(\begin{bmatrix}
    1
\end{bmatrix}\) for each vertices in \(U\). The remaining edges are introduced through a
perturbation, while the diagonal entries corresponding to \(W\) are adjusted using the
Implicit Function Theorem so that the required principal minors remain
zero.

To cover the vertices in \(U\), we assign each vertex of \(U\) to one
of its neighbours in \(W\). This produces a collection of stars, for
which appropriate non-singular block matrices are constructed. A second
application of the Implicit Function Theorem then yields a matrix in
\(S(G)\) for which every vertex of \(U\) is a \(P\)-vertex.

As an immediate consequence of the first construction, we extend the
main perfect-matching result of Sharma and Panda from bipartite graphs
to arbitrary graphs. Namely, we prove that
\[
G\text{ has a perfect matching}
\quad\Longrightarrow\quad
G\text{ has property }(P)
\]
for every graph \(G\), without any bipartiteness assumption.

The paper is organized as follows. We first introduce the notions of \(P\)-vertex covering and the \(P\)-vertex covering number. We then show that any graph containing an isolated vertex cannot have finite covering number. Using maximal matchings and the Implicit Function Theorem, we prove that every graph without isolated vertices has covering number at most two. We next specialize this construction to graphs admitting a perfect matching and establish that every such graph has property $(P)$. Finally, we present examples illustrating how to explicitly construct matrices whose \(P\)-vertices collectively cover all the vertices of a graph $G$.

\section{The P vertex covering number of a graph G}

Throughout, we assume that every graph on \(n\) vertices has vertex set
\([n]=\{1,\ldots,n\}\).

\begin{defn}
Let \(A=(a_{ij})\) be an \(n\times n\) real symmetric matrix. The
\emph{support graph} of \(A\), denoted by \(G(A)\), is the graph with
vertex set \([n]\) and edge set
\[
E(G(A))
=
\bigl\{\{i,j\}: i\neq j \text{ and } a_{ij}\neq 0\bigr\}.
\]
\end{defn}

Given a graph \(G\) on \(n\) vertices, define
\[
S(G)
=
\bigl\{A\in\mathbb{R}^{n\times n}: A^{T}=A
\text{ and } G(A)=G\bigr\}.
\]

For an \(n\times n\) matrix \(A\) and \(i\in[n]\), let \(A(i)\) denote
the principal submatrix obtained by deleting the \(i^\text{th}\) row and the
\(i^\text{th}\) column of \(A\).

By the Cauchy interlacing theorem, for every \(i\in[n]\) and every
\(\lambda\in\mathbb{R}\),
\[
m_{A(i)}(\lambda)-m_A(\lambda)\in\{-1,0,1\},
\]
where \(m_A(\lambda)\) denotes the algebraic multiplicity of \(\lambda\)
as an eigenvalue of \(A\).

\begin{defn}
Let \(A\in S(G)\). A vertex \(i\in V(G)\) is called a
\emph{\(P\)-vertex of \(A\)} if
\[
m_{A(i)}(0)-m_A(0)=1.
\]
\end{defn}

\begin{defn}
A graph \(G\) is said to have the \emph{full \(P\)-vertex property}, or
\emph{property \((P)\)}, if there exists a matrix \(A\in S(G)\) for
which every vertex of \(G\) is a \(P\)-vertex.
\end{defn}

\begin{rem}
A singular real symmetric matrix of order \(n\) can have at most \(n-1\) \(P\)-vertices. Indeed, if \(A\) is singular with rank \(r\), then \(A\) has a non-singular principal submatrix of order \(r\) since \(A\) is symmetric. Choosing a vertex outside this principal submatrix, the corresponding principal submatrix \(A(i)\) has the same rank \(r\). Hence,
\[
m_{A(i)}(0)=(n-1)-r=m_A(0)-1,
\]
so \(i\) is not a \(P\)-vertex. Consequently, any matrix witnessing property \((P)\)
must be non-singular. We shall therefore restrict our attention to
non-singular matrices.
\end{rem}


\begin{defn}
A graph \(G\) is said to admit a \emph{\(P\)-vertex covering} if, for
every vertex \(i\in V(G)\), there exists a non-singular matrix
\(A_i\in S(G)\) such that \(i\) is a \(P\)-vertex of \(A_i\).
\end{defn}

\begin{defn}
Let \(G\) be a graph. The \emph{\(P\)-vertex covering number} of \(G\),
denoted by \(p(G)\), is defined by
\[
p(G)
=
\min\left\{
|\mathcal{A}|:
\begin{array}{l}
\mathcal{A}\subseteq S(G),\ \text{every matrix in }\mathcal{A}
\text{ is non-singular, and}\\[2mm]
\text{every vertex of }G\text{ is a \(P\)-vertex of some }
A\in\mathcal{A}
\end{array}
\right\},
\]
provided that such a family \(\mathcal{A}\) exists. If no such family
exists, we set \(p(G)=\infty\).
\end{defn}

\textbf{Note:}
If \(G\) admits a \(P\)-vertex covering, then
\(
p(G)\leq |V(G)|.
\)
Moreover, if \(G\) has property \((P)\), then \(p(G)=1\).

\begin{lem}\label{isolated}
If a graph \(G\) has an isolated vertex, then \(p(G)=\infty\).
\end{lem}

\begin{proof}
Let \(i\) be an isolated vertex of \(G\) and let \(A\in S(G)\). Since
\(i\) is isolated, every off-diagonal entry in the \(i^\text{th}\) row and
column of \(A\) is zero. After applying the permutation which permutes $i$ to $1$ to the
rows and the same permutation to the columns, we may write
\[
A=
\begin{bmatrix}
a_{ii} & \mathbf{0}^{T}\\
\mathbf{0} & B
\end{bmatrix},
\]
where \(B\in S(G-i)\). 

Moreover, \(A(i)\) is permutation-similar to
\(B\) and hence
\(
m_{A(i)}(0)=m_B(0).
\)

If \(a_{ii}\neq 0\), then
\(
m_A(0)=m_B(0),
\) whereas if \(a_{ii}=0\), then
\(
m_A(0)=m_B(0)+1.
\)

Therefore,
\[
m_{A(i)}(0)-m_A(0)
=
\begin{cases}
0, & a_{ii}\neq 0,\\
-1, & a_{ii}=0.
\end{cases}
\]
In either case, \(i\) is not a \(P\)-vertex of \(A\).

Thus, no matrix in \(S(G)\) can have \(i\) as a \(P\)-vertex.
Consequently, \(G\) does not admit a \(P\)-vertex covering and hence
\(
p(G)=\infty.
\)

\end{proof}
The next theorem completely characterizes the graphs admitting a \(P\)-vertex covering. Combined with Lemma~\ref{isolated}, it shows that a graph admits a $P$-vertex covering if and only if it has no isolated vertices.


\begin{theorem}\label{thm:main}
Let \(G\) be a graph with no isolated vertices. Then
\(
p(G)\leq 2.
\)
\end{theorem}

\begin{proof}
Let \(M\) be a maximal matching of \(G\). Let \(W\) be the set of
vertices saturated by \(M\) and let
\(
U=V(G)\setminus W.
\)
Since \(M\) is maximal, \(U\) is an independent set. Indeed, if two
vertices of \(U\) were adjacent, then their joining edge could be added
to \(M\), contradicting the maximality of \(M\).

We construct a non-singular matrix in \(S(G)\) for which every vertex
of \(W\) is a \(P\)-vertex and when \(U\neq\varnothing\), a second
non-singular matrix in \(S(G)\) for which every vertex of \(U\) is a
\(P\)-vertex.

Throughout the proof, for distinct vertices \(i\) and \(j\), we write
\(A(i,j)\) for the principal submatrix obtained by deleting the rows
and columns indexed by \(i\) and \(j\).

\medskip

\noindent
\textbf{Step 1: A matrix for the vertices in \(W\).}

Suppose
\(
M=\bigl\{\{x_1,y_1\},\ldots,\{x_t,y_t\}\bigr\}.
\)
Thus,
\(
W=\{x_1,y_1,\ldots,x_t,y_t\}.
\)
Order the vertices so that the vertices in \(W\) occur first, with the
endpoints of each matching edge consecutive, followed by the vertices
in \(U\).

Let
\(
B=
\begin{bmatrix}
0&1\\
1&0
\end{bmatrix}
\)
and define the block diagonal matrix
\(
A_W^0=
\operatorname{diag}
\bigl(
\underbrace{B,\ldots,B}_{t\text{ copies}},
I_{|U|}
\bigr).
\)
Since \(\det B=-1\), the matrix \(A_W^0\) is non-singular.

For each \(v\in W\), let \(d_v\) be a variable corresponding to the
\(v^\text{th}\) diagonal entry, and write
\[
\mathbf d=(d_v)_{v\in W}\in\mathbb R^{|W|}.
\]

Let \(\epsilon\in\mathbb R\). Define \(A_W(\mathbf d,\epsilon)\) to be
the symmetric matrix whose entries are given as follows:

\[
\bigl(A_W(\mathbf d,\epsilon)\bigr)_{vv}
=
\begin{cases}
d_v, & v\in W,\\
1,   & v\in U.
\end{cases}
\]
and for distinct vertices \(v,z\),
\[
\bigl(A_W(\mathbf d,\epsilon)\bigr)_{vz}
=
\begin{cases}
1,
& \{v,z\}\in M,\\
\epsilon,
& \{v,z\}\in E(G)\setminus M,\\
0,
& \{v,z\}\notin E(G).
\end{cases}
\]

Let \(\mathbf d^0=\mathbf 0\). Then
\(
A_W(\mathbf d^0,0)=A_W^0.
\)
Define
\(
F_W:\mathbb R^{|W|}\times\mathbb R
\longrightarrow
\mathbb R^{|W|}
\)
by
\(
F_W(\mathbf d,\epsilon)
=
\bigl(
\det A_W(\mathbf d,\epsilon)(v)
\bigr)_{v\in W}.
\)
Since the entries of \(A_W(\mathbf d,\epsilon)\) depend polynomially
on \((\mathbf d,\epsilon)\), the map \(F_W\) is continuously
differentiable.

For every \(v\in W\), deleting the row and column indexed by \(v\)
from \(A_W^0\) leaves a \(1\times1\) zero block arising from the
matching edge containing \(v\). Therefore,
\(
F_W(\mathbf d^0,0)=\mathbf 0.
\)

We now compute the Jacobian of \(F_W\) with respect to \(\mathbf d\) at the point \((\mathbf{d}^0,0)\).

For \(v,z\in W\),
\[
\left.
\frac{\partial}{\partial d_z}
\det A_W(\mathbf d,\epsilon)(v)
\right|_{(\mathbf d^0,0)}
=
\begin{cases}
0, & v=z,\\[1mm]
\det A_W^0(v,z), & v\neq z.
\end{cases}
\]
The first equality holds because the entry \(d_v\) does not occur in
\(A_W(\mathbf d,\epsilon)(v)\). 
The second equality follows from the calculation that if \(v\neq z\), then \(d_z\) appears only as the \((z,z)\)-entry of
\(A_W(\mathbf{d},\epsilon)(v)\). In the cofactor expansion of
\(\det\bigl(A_W(\mathbf{d},\epsilon)(v)\bigr)\) along the row corresponding to \(z\), the term
corresponding to the \((z,z)\)-entry contributes
\(
d_z\cdot\det\bigl(A_W(\mathbf{d},\epsilon)(v,z)\bigr),
\)
whereas all the other terms involve entries of that row outside the
\(z^{\text{th}}\) column and hence do not contain \(d_z\).

If \(v\) and \(z\) are the endpoints of the same edge of \(M\), then
deleting \(v\) and \(z\) removes one entire copy of \(B\). Hence,
\[
\det A_W^0(v,z)=(-1)^{t-1}\neq0.
\]

If \(v\) and \(z\) belong to two different matching edges, then
\(A_W^0(v,z)\) contains two \(1\times1\) zero diagonal blocks arising from two different matching edges containing $v$ and $z$  and
therefore
\(
\det A_W^0(v,z)=0.
\)

Consequently, with respect to the ordering
\(
x_1,y_1,x_2,y_2,\ldots,x_t,y_t,
\)
the Jacobian is a block diagonal matrix given by,
\[
D_{\mathbf d}F_W(\mathbf d^0,0)
=
\operatorname{diag}
\left(
\underbrace{
\begin{bmatrix}
0&c_W\\
c_W&0
\end{bmatrix},
\ldots,
\begin{bmatrix}
0&c_W\\
c_W&0
\end{bmatrix}
}_{t\text{ copies}}
\right),
\]
where
\(
c_W=(-1)^{t-1}.
\)
Every block is non-singular and hence
\(
D_{\mathbf d}F_W(\mathbf d^0,0)
\)
is non-singular.

By the Implicit Function Theorem, there exist \(\delta_W>0\) and a unique
continuously differentiable function
\(
\varphi_W:(-\delta_W,\delta_W)\longrightarrow\mathbb R^{|W|}
\)
such that
\(
\varphi_W(0)=\mathbf d^0
\)
and
\(
F_W(\varphi_W(\epsilon),\epsilon)=\mathbf 0
\)
for every \(\epsilon\in(-\delta_W,\delta_W)\).
Since
\(
\det A_W(\mathbf d^0,0)=\det A_W^0\neq0,
\)
continuity of the determinant allows us, after decreasing \(\delta_W\)
if necessary, to assume that
\(
\det A_W(\varphi_W(\epsilon),\epsilon)\neq0
\)
for every \(\epsilon\in(-\delta_W,\delta_W)\).

Choose
\(
\epsilon_W\in(-\delta_W,\delta_W)\setminus\{0\}
\)
and set
\(
\mathcal A_W
=
A_W(\varphi_W(\epsilon_W),\epsilon_W).
\)
Every edge of \(M\) has matrix entry \(1\), every edge of
\(E(G)\setminus M\) has matrix entry \(\epsilon_W\neq0\), and every
non-edge has matrix entry zero. Hence,
\(
\mathcal A_W\in S(G).
\)
Moreover,
\(
\det\mathcal A_W\neq0
\)
and, for every \(v\in W\),
\(
\det\mathcal A_W(v)=0.
\)

Thus,
\(
m_{\mathcal A_W}(0)=0
\)
and
\(
m_{\mathcal A_W(v)}(0)\geq1.
\)
By the Cauchy interlacing theorem,
\[
m_{\mathcal A_W(v)}(0)-m_{\mathcal A_W}(0)\leq1.
\]
It follows that
\(
m_{\mathcal A_W(v)}(0)=1
\)
and therefore every \(v\in W\) is a \(P\)-vertex of
\(\mathcal A_W\).

If \(U=\varnothing\), then \(W=V(G)\), so \(\mathcal A_W\) already
shows that \(p(G)=1\). We may therefore assume that
\(U\neq\varnothing\).

\medskip

\noindent
\textbf{Step 2: A matrix for the vertices in \(U\).}

Since \(G\) has no isolated vertices and \(U\) is independent, every
vertex of \(U\) has a neighbour in \(W\). For each \(u\in U\), choose
one such neighbour and denote it by \(f(u)\).

For \(w\in W\), define
\(
U_w=\{u\in U:f(u)=w\}.
\)
Let
\(
W'=\{w\in W:U_w\neq\varnothing\}
=
\{w_1,\ldots,w_q\}.
\)
For each \(1\leq i\leq q\), write
\(
U_i=U_{w_i}
=
\{u_{i,1},\ldots,u_{i,n_i}\},
\
n_i=|U_i|\geq1.
\)
The sets \(U_1,\ldots,U_q\) form a partition of \(U\).

For each \(i\), define the \((n_i+1)\times(n_i+1)\) matrix
\[
S_i=
\begin{bmatrix}
n_i-1&\mathbf 1_{n_i}^{T}\\
\mathbf 1_{n_i}&I_{n_i}
\end{bmatrix},
\]
where the first row and column correspond to \(w_i\), and the
remaining rows and columns correspond to the vertices in \(U_i\).

By the Schur complement formula,
\[
\det S_i
=
\det(I_{n_i})
\left(
n_i-1-\mathbf 1_{n_i}^{T}\mathbf 1_{n_i}
\right)
=(n_i-1)-n_i
=-1.
\]
Thus every \(S_i\) is non-singular.

If \(u\in U_i\), then deleting the row and column corresponding to 
\(u\) from the matrix \(S_i\) gives,
\[
\det S_i(u)
=
(n_i-1)-(n_i-1)
=0.
\]

If \(n_i\geq2\) and \(u,v\in U_i\) are distinct, then
\[
\det S_i(u,v)
=
(n_i-1)-(n_i-2)
=1.
\]

With respect to the ordering
\(
w_1,U_1,w_2,U_2,\ldots,w_q,U_q,W\setminus W',
\)
define
\[
A_U^0
=
\operatorname{diag}
\left(
S_1,\ldots,S_q,I_{|W\setminus W'|}
\right).
\]

Since every diagonal block is non-singular,
\(
\det A_U^0\neq0.
\)
Moreover, for every \(u\in U\),
\(
\det A_U^0(u)=0.
\)

We now choose the diagonal variables used in the Implicit Function
Theorem. Define a map
\(
\psi:U\longrightarrow V(G)
\)
as follows. If \(u\in U_i\), set
\(
\psi(u)=
\begin{cases}
w_i, & n_i=1,\\
u,   & n_i\geq2
\end{cases}
\).
The map \(\psi\) is injective. 

Thus, we may assign one independent
diagonal variable \(h_u\) to the diagonal entry indexed by
\(\psi(u)\), for each \(u\in U\). Write
\(
\mathbf h=(h_u)_{u\in U}\in\mathbb R^{|U|}.
\)
Let \(\mathbf h^0\) denote the values of these diagonal entries in
\(A_U^0\). Explicitly,
\[
h_u^0=
\begin{cases}
0, & u\in U_i\text{ and }n_i=1,\\
1, & u\in U_i\text{ and }n_i\geq2
\end{cases}.
\]

Let
\(
E_0=\bigl\{\{w_i,u\}:1\leq i\leq q,\ u\in U_i\bigr\}.
\)
Define \(A_U(\mathbf h,\epsilon)\) as follows. Its diagonal entries
indexed by \(\psi(u)\), \(u\in U\), are replaced by the corresponding
variables \(h_u\), while all other diagonal entries remain equal to
their values in \(A_U^0\). For distinct vertices \(v,z\), define
\[
\bigl(A_U(\mathbf h,\epsilon)\bigr)_{vz}
=
\begin{cases}
1,
& \{v,z\}\in E_0,\\
\epsilon,
& \{v,z\}\in E(G)\setminus E_0,\\
0,
& \{v,z\}\notin E(G).
\end{cases}
\]

Then
\(
A_U(\mathbf h^0,0)=A_U^0.
\)
Define
\(
F_U:\mathbb R^{|U|}\times\mathbb R
\longrightarrow
\mathbb R^{|U|}
\)
by
\[
F_U(\mathbf h,\epsilon)
=
\bigl(
\det A_U(\mathbf h,\epsilon)(u)
\bigr)_{u\in U}
\]
This is a continuously differentiable map and
\(
F_U(\mathbf h^0,0)=\mathbf 0.
\)

We compute the Jacobian of \(F_U\) with respect to \(\mathbf h\).
For \(u,v\in U\),
\[
\left.
\frac{\partial}{\partial h_v}
\det A_U(\mathbf h,\epsilon)(u)
\right|_{(\mathbf h^0,0)}
=
\begin{cases}
0,
& \psi(v)=u,\\[1mm]
\det A_U^0(u,\psi(v)),
& \psi(v)\neq u.
\end{cases}
\]

Let
\(
c_U=(-1)^{q-1}.
\)
We examine the Jacobian block corresponding to each set \(U_i\).

Suppose first that \(n_i=1\), and write \(U_i=\{u\}\). Then
\(
\psi(u)=w_i,
\)
and therefore
\[
\left.
\frac{\partial}{\partial h_u}
\det A_U(\mathbf h,\epsilon)(u)
\right|_{(\mathbf h^0,0)}
=
\det A_U^0(u,w_i).
\]

Deleting both \(u\) and \(w_i\) removes the entire block \(S_i\).
All the remaining blocks are non-singular and hence
\[
\det A_U^0(u,w_i)=c_U\neq0.
\]
Thus, the Jacobian block corresponding to \(U_i\) is the
\(1\times1\) matrix
\(
[c_U].
\)

Now suppose that \(n_i\geq2\). For \(u,v\in U_i\), we have
\(\psi(v)=v\). Therefore,
\[
\left.
\frac{\partial}{\partial h_v}
\det A_U(\mathbf h,\epsilon)(u)
\right|_{(\mathbf h^0,0)}
=
\begin{cases}
0, & u=v,\\
c_U, & u\neq v.
\end{cases}
\]
Indeed, when \(u\neq v\), deleting the two leaves \(u\) and \(v\)
from \(S_i\) gives a block of determinant \(1\), while every other
block contributes its determinant.

If \(u\in U_i\) and \(v\in U_j\) with \(i\neq j\), then
\(
\det A_U^0(u,\psi(v))=0.
\)
To see this, observe that deleting \(u\) makes the block \(S_i\)
singular, and deleting \(\psi(v)\), which belongs to a different
block, does not remove this singularity.

Consequently, after ordering the vertices of \(U\) according to the
partition
\(
U=U_1\mathbin{\dot\cup}\cdots\mathbin{\dot\cup}U_q,
\)
the Jacobian is a block diagonal matrix given by,
\[
D_{\mathbf h}F_U(\mathbf h^0,0)
=
\operatorname{diag}(J_1,\ldots,J_q),
\]
where
\[
J_i=
\begin{cases}
[c_U], & n_i=1,\\[2mm]
c_U\bigl(\mathbf J_{n_i}-I_{n_i}\bigr), & n_i\geq2,
\end{cases}
\]
and \(\mathbf J_{n_i}\) denotes the \(n_i\times n_i\) all-one
matrix.

For \(n_i\geq2\), the eigenvalues of
\(\mathbf J_{n_i}-I_{n_i}\) are
\(
n_i-1
\text{ and } 
-1
\)
with multiplicities \(1\) and \(n_i-1\), respectively. Hence every
\(J_i\) is non-singular. Therefore,
\(
D_{\mathbf h}F_U(\mathbf h^0,0)
\)
is non-singular.

By the Implicit Function Theorem, there exist \(\delta_U>0\) and a unique
continuously differentiable function
\(
\varphi_U:(-\delta_U,\delta_U)\longrightarrow\mathbb R^{|U|}
\)
such that
\(
\varphi_U(0)=\mathbf h^0
\)
and
\(
F_U(\varphi_U(\epsilon),\epsilon)=\mathbf 0
\)
for every \(\epsilon\in(-\delta_U,\delta_U)\).
Since
\(\det A_U(\mathbf h^0,0)=\det A_U^0\neq0,
\)
we may decrease \(\delta_U\), if necessary, so that
\(
\det A_U(\varphi_U(\epsilon),\epsilon)\neq0
\)
for every \(\epsilon\in(-\delta_U,\delta_U)\).
Choose
\(
\epsilon_U\in(-\delta_U,\delta_U)\setminus\{0\}
\)
and set
\(
\widetilde{\mathcal{A}}_U
=
A_U(\varphi_U(\epsilon_U),\epsilon_U).
\)
By construction,
\(
\widetilde{\mathcal{A}}_U\in S(G)
\text{ and }
\det\widetilde{\mathcal{A}}_U\neq0.
\)
Furthermore, for every \(u\in U\),
\(
\det\widetilde{\mathcal{A}}_U(u)=0.
\)

Since the above construction is carried out with respect to a temporary ordering of the vertices, the matrix obtained need not correspond to the original labeling of \(G\) used in Step 1. Let \(P\) be the permutation matrix that restores the original ordering in Step 1 and define
\(
\mathcal{A}_U=P^{T}\widetilde{\mathcal{A}}_U P.
\)
Since simultaneous permutation of rows and columns preserves symmetry, non-singularity, principal minors, and the graph of a matrix, it follows that \(\mathcal{A}_U\in S(G)\), \(\det(\mathcal{A}_U)\neq0\), and \(\det(\mathcal{A}_U(u))=0\) for every \(u\in U\).
As in Step 1, Cauchy interlacing theorem implies that
\(
m_{\widetilde{\mathcal{A}}_U(u)}(0)=1
\)
for every \(u\in U\). Thus, every vertex of \(U\) is a \(P\)-vertex
of \(\mathcal{A}_U\).

The matrix \(\mathcal A_W\) covers all vertices of \(W\) makes all the vertices of $W$ as $P$-vertices and the
matrix \(\mathcal A_U\) covers all vertices of \(U\) makes all the vertices of $U$ as $P$-vertices. Therefore,
\(
p(G)\leq2.
\) 

\end{proof}

\begin{rem}
    The proof of the above theorem still works if we choose a matching which saturates a dominating set, that is, every unsaturated vertex is adjacent to atleast one of the saturated vertex. 
\end{rem}


The following corollary is an immediate consequence of the proof of Theorem~\ref{thm:main}. It extends the corresponding result of Sharma and Panda \cite[Theorem 3.1]{sharma} from bipartite graphs to arbitrary graphs.

\begin{cor}
If a graph \(G\) has a perfect matching, then
\(
p(G)=1.
\)
\end{cor}

\begin{proof}
Let \(M\) be a perfect matching of \(G\). Then every vertex of \(G\)
is saturated by \(M\), so in the notation of the preceding proof,
\(
W=V(G)
\text{ and }
U=\varnothing.
\)
The construction in Step 1 therefore produces a non-singular matrix
\(\mathcal A_W\in S(G)\) for which every vertex of \(G\) is a
\(P\)-vertex. Hence \(
p(G)=1
\) and consequently \(G\) has property \((P)\).

\end{proof}

Although the proof of Theorem~\ref{thm:main} is constructive, the resulting implicit function is generally difficult to compute explicitly. The following example illustrates the construction and shows that, in special cases, the implicit function can be determined exactly.
\section{Examples}
\begin{example}
\label{ex:1}
Consider the graph \(G\) shown in Figure~\ref{fig:1}.

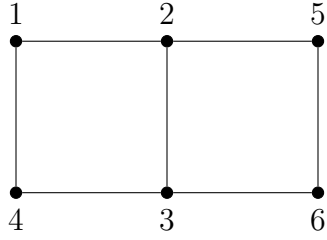
\begin{figure}[H]
    \centering
    \begin{tikzpicture}
        \node[circle,draw,fill=black,inner sep=1.5pt,label=above:\(1\)]
            (1) at (0,0) {};
        \node[circle,draw,fill=black,inner sep=1.5pt,label=above:\(2\)]
            (2) at (2,0) {};
        \node[circle,draw,fill=black,inner sep=1.5pt,label=above:\(5\)]
            (5) at (4,0) {};
        \node[circle,draw,fill=black,inner sep=1.5pt,label=below:\(4\)]
            (4) at (0,-2) {};
        \node[circle,draw,fill=black,inner sep=1.5pt,label=below:\(3\)]
            (3) at (2,-2) {};
        \node[circle,draw,fill=black,inner sep=1.5pt,label=below:\(6\)]
            (6) at (4,-2) {};

        \draw (1)--(2)--(5);
        \draw (4)--(3)--(6);
        \draw (1)--(4);
        \draw (2)--(3);
        \draw (5)--(6);
    \end{tikzpicture}
    \caption{The graph in Example~\ref{ex:1}.}
    \label{fig:1}
\end{figure}

Consider the perfect matching
\(
M=\bigl\{\{1,2\},\{3,4\},\{5,6\}\bigr\}.
\)
The corresponding initial matrix is
\[
A_0=
\begin{bmatrix}
0&1&0&0&0&0\\
1&0&0&0&0&0\\
0&0&0&1&0&0\\
0&0&1&0&0&0\\
0&0&0&0&0&1\\
0&0&0&0&1&0
\end{bmatrix}.
\]

Let
\(
\mathbf d=(d_1,\ldots,d_6)\in\mathbb R^6 \text{ and } \epsilon \in \R
\)
and define 
\[
A(\mathbf d,\epsilon):=
\begin{bmatrix}
d_1&1&0&\epsilon&0&0\\
1&d_2&\epsilon&0&\epsilon&0\\
0&\epsilon&d_3&1&0&\epsilon\\
\epsilon&0&1&d_4&0&0\\
0&\epsilon&0&0&d_5&1\\
0&0&\epsilon&0&1&d_6
\end{bmatrix}.
\]
Then
\(
A(\mathbf 0,0)=A_0.
\)

Define
\(
F:\mathbb R^6\times\mathbb R\longrightarrow\mathbb R^6
\)
by
\(
F(\mathbf d,\epsilon)
=
\bigl(
\det A(\mathbf d,\epsilon)(1),\ldots,
\det A(\mathbf d,\epsilon)(6)
\bigr).
\)
Clearly,
\(
F(\mathbf 0,0)=\mathbf 0.
\)

The Jacobian of \(F\) with respect to \(\mathbf d\), evaluated at
\((\mathbf 0,0)\), is
\[
D_{\mathbf d}F(\mathbf 0,0)
=
\begin{bmatrix}
0&1&0&0&0&0\\
1&0&0&0&0&0\\
0&0&0&1&0&0\\
0&0&1&0&0&0\\
0&0&0&0&0&1\\
0&0&0&0&1&0
\end{bmatrix}
=A_0.
\]
Since the Jacobian is non-singular, the Implicit Function Theorem yields a
neighborhood \(I\) of \(0\) and a unique continuously differentiable
function
\(
\varphi:I\longrightarrow\mathbb R^6
\)
such that
\(
\varphi(0)=\mathbf 0
\)
and
\(
F(\varphi(\epsilon),\epsilon)=\mathbf 0
\)
for every \(\epsilon\in I\).
In fact, the implicit function can be identified directly. 

The matrix
\(A(\mathbf 0,\epsilon)\) is a weighted bipartite adjacency matrix with
bi-partition
\(
\{1,3,5\}\mathbin{\dot\cup}\{2,4,6\}.
\)
After deleting any vertex, the two parts have cardinalities \(2\) and
\(3\). Consequently, for each \(i\in\{1,\ldots,6\}\), the matrix
\(A(\mathbf 0,\epsilon)(i)\) is permutation-similar to a matrix of the
form
\(
\begin{bmatrix}
0&B\\
B^T&0
\end{bmatrix},
\)
where \(B\) is either a \(2\times 3\) or a \(3\times2\) rectangular matrix. Accordingly,
\(
\operatorname{rank}\!\left(A(\mathbf 0,\epsilon)(i)\right)
=
2\operatorname{rank}(B)
\leq
2\min\{2,3\}
=
4
<
5
\)
and therefore \(A(\mathbf 0,\epsilon)(i)\) is singular. Hence,
\(
\det\!\left(A(\mathbf 0,\epsilon)(i)\right)=0
\)
for every \(i\in\{1,\ldots,6\}\). Thus,
\(
F(\mathbf 0,\epsilon)=\mathbf 0
\)
for every \(\epsilon\in\mathbb R\).

Hence the implicit function \(\varphi\) is zero function. Indeed,
\(
\varphi(0)=\mathbf 0
\)
and
\(
F(\varphi(\epsilon),\epsilon)=\mathbf 0
\)
for every \(\epsilon\in\mathbb R\). By the local uniqueness furnished by
the Implicit Function Theorem, this is precisely the implicit function.

A direct computation gives
\(
\det A(\mathbf 0,\epsilon)
=
-\bigl(\epsilon^3-\epsilon^2+1\bigr)^2.
\)
Taking \(\epsilon=1\), we obtain
\(
\det A(\mathbf 0,1)=-1\neq0.
\)
Thus, the matrix
\[
\mathcal A=A(\mathbf 0,1)
=
\begin{bmatrix}
0&1&0&1&0&0\\
1&0&1&0&1&0\\
0&1&0&1&0&1\\
1&0&1&0&0&0\\
0&1&0&0&0&1\\
0&0&1&0&1&0
\end{bmatrix}
\]
is non-singular and belongs to \(S(G)\). Moreover,
\(
\det\mathcal A(i)=0
\)
for every \(i\in V(G)\). Since \(\mathcal A\) is non-singular, the
Cauchy interlacing theorem implies that
\(
m_{\mathcal A(i)}(0)=1
\)
for every \(i\in V(G)\). Hence, every vertex of \(G\) is a
\(P\)-vertex of \(\mathcal A\).
Therefore, \(G\) has property \((P)\). 
\end{example}
\begin{rem}
   The preceding example illustrates the constructive nature of our approach. In this case, the implicit function is identically zero and hence no perturbation of the diagonal entries is required. Since we have chosen \(\epsilon=1\), the matrix produced by our construction is precisely the adjacency matrix of \(G\). More generally, choosing any nonzero value of \(\epsilon\) for which the resulting matrix is non-singular yields a weighted adjacency matrix witnessing property \((P)\). We also remark that the matrix obtained here differs from the one constructed by Sharma and Panda for the same graph, illustrating that a graph may admit several distinct matrices witnessing property \((P)\). 
\end{rem}
\begin{example}\label{ex:2}
We illustrate the construction in the proof of Theorem~\ref{thm:main}
using the graph shown in Figure~\ref{fig:2}.

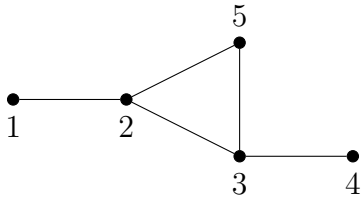
\begin{figure}[H]
    \centering
    \begin{tikzpicture}
        \node[circle,draw,fill=black,inner sep=1.5pt,label=below:\(1\)]
            (1) at (0,0) {};
        \node[circle,draw,fill=black,inner sep=1.5pt,label=below:\(2\)]
            (2) at (1.5,0) {};
        \node[circle,draw,fill=black,inner sep=1.5pt,label=below:\(3\)]
            (3) at (3,-0.75) {};
        \node[circle,draw,fill=black,inner sep=1.5pt,label=below:\(4\)]
            (4) at (4.5,-0.75) {};
        \node[circle,draw,fill=black,inner sep=1.5pt,label=above:\(5\)]
            (5) at (3,0.75) {};

        \draw (1)--(2)--(3)--(4);
        \draw (2)--(5);
        \draw (5)--(3);
    \end{tikzpicture}
    \caption{The graph in Example~\ref{ex:2}.}
    \label{fig:2}
\end{figure}

Let
\(
M=\bigl\{\{1,2\},\{3,4\}\bigr\}
\)
be a maximal matching of \(G\). Then
\(
W=\{1,2,3,4\}
\)
is the set of vertices saturated by \(M\) and
\(
U=\{5\}
\)
is the set of unsaturated vertices.

We first apply the construction from the proof of
Theorem~\ref{thm:main} to obtain a matrix for which every vertex in
\(W\) is a \(P\)-vertex.

Let
\(
\mathbf d=(d_1,d_2,d_3,d_4)\in\mathbb R^4
\)
and let \(\epsilon\in\mathbb R\). Consider the matrix,
\[
\overline A_W=A_W(\mathbf d,\epsilon)
=
\begin{bmatrix}
d_1&1&0&0&0\\
1&d_2&\epsilon&0&\epsilon\\
0&\epsilon&d_3&1&\epsilon\\
0&0&1&d_4&0\\
0&\epsilon&\epsilon&0&1
\end{bmatrix}.
\]

A direct computation gives
\[
\det(\overline A_W)
=
1-d_1d_2-d_3d_4+d_1d_2d_3d_4
+\epsilon^2
\bigl(
d_1+d_4-d_1d_4-d_1d_2d_4-d_1d_3d_4
\bigr)
+2\epsilon^3d_1d_4.
\]

Define
\(
F_W:\mathbb R^4\times\mathbb R\longrightarrow\mathbb R^4
\)
by
\(
F_W(\mathbf d,\epsilon)
=
\bigl(
\det(\overline A_W(1)),
\ldots,
\det(\overline A_W(4))
\bigr).
\)\\
Clearly,
\(
F_W(\mathbf 0,0)=\mathbf 0.
\)

Moreover, the Jacobian of \(F_W\) with respect to \(\mathbf d\), evaluated
at \((\mathbf 0,0)\), is
\[
D_{\mathbf d}F_W(\mathbf 0,0)
=
\begin{bmatrix}
0&-1&0&0\\
-1&0&0&0\\
0&0&0&-1\\
0&0&-1&0
\end{bmatrix},
\]
which is non-singular.

Hence, by the Implicit Function Theorem, there exist a neighbourhood
\(I\) of \(0\) and a unique continuously differentiable function
\(
\varphi:I\longrightarrow\mathbb R^4
\)
such that
\(
\varphi(0)=\mathbf 0
\)
and
\(
F_W(\varphi(\epsilon),\epsilon)=\mathbf 0
\)
for every \(\epsilon\in I\). In this example, the implicit function can
be determined explicitly.

The relevant principal minors are
\begin{align}
\det(\overline A_W(2))
&=
d_1\bigl[d_4(d_3-\epsilon^2)-1\bigr],
\label{eq:ex3-f2}
\\
\det(\overline A_W(3))
&=
d_4\bigl[d_1(d_2-\epsilon^2)-1\bigr].
\label{eq:ex3-f3}
\end{align}

Define
\(
g:\mathbb R^4\times\mathbb R\longrightarrow\mathbb R
\)
by
\(
g(\mathbf d,\epsilon)
=
d_4(d_3-\epsilon^2)-1.
\)
Then
\(
\det(\overline A_W(2))=d_1\cdot g(\mathbf d,\epsilon),
\)
and
\(
g(\mathbf 0,0)=-1\neq0.
\)

Since \(g\) is continuous, there exists a neighbourhood \(N\) of
\((\mathbf 0,0)\) such that
\(\
g(\mathbf d,\epsilon)\neq0
\)
for every \((\mathbf d,\epsilon)\in N\). Consequently, in order that
\(
F_W(\mathbf d,\epsilon)=\mathbf 0,
\)
it is necessary that
\(
d_1=0.
\)

An identical argument applied to \eqref{eq:ex3-f3} shows that
\(
d_4=0.
\)

Substituting \(d_1=d_4=0\) 
in \(\overline{A}_W\), we get
\[
\det(\overline A_W(1))
=
-d_2+\epsilon^2, \qquad
\det(\overline A_W(4))
=
-d_3+\epsilon^2.
\]
Therefore,
\(
d_2=d_3=\epsilon^2
\) since we need \(F_W(\mathbf{d},\epsilon)=\mathbf{0}\).

Hence the implicit function is given by
\(
\varphi(\epsilon)
=
(0,\epsilon^2,\epsilon^2,0).
\)

Indeed,
\(
\varphi(0)=\mathbf 0
\)
and
\(
F_W(\varphi(\epsilon),\epsilon)=\mathbf 0
\)
for every \(\epsilon\in\mathbb R\). By the local uniqueness furnished by
the Implicit Function Theorem, this is precisely the implicit function
obtained above.

Substituting
\(
\mathbf d=\varphi(\epsilon)
\)
into \(\overline A_W\), we obtain
\(
\det A_W(\varphi(\epsilon),\epsilon)=1,
\)
independently of \(\epsilon\). Consequently, every nonzero choice of
\(\epsilon\) produces a non-singular matrix in \(S(G)\) for which every
vertex in \(W\) is a \(P\)-vertex. 

Choosing
\(
\epsilon=1,
\)
we obtain
\[
\mathcal A_W
=
\begin{bmatrix}
0&1&0&0&0\\
1&1&1&0&1\\
0&1&1&1&1\\
0&0&1&0&0\\
0&1&1&0&1
\end{bmatrix}.
\]

We now apply the second part of the construction to obtain a matrix for
which the vertex in \(U\) is a \(P\)-vertex.

Let \(f(5)=3\), hence \(U_3=\{5\}\) and let \(h,\epsilon\in\mathbb R\). Consider the matrix,
\[
\overline A_U=A_U(h,\epsilon)
=
\begin{bmatrix}
1&\epsilon&0&0&0\\
\epsilon&1&\epsilon&0&\epsilon\\
0&\epsilon&h&\epsilon&1\\
0&0&\epsilon&1&0\\
0&\epsilon&1&0&1
\end{bmatrix}.
\]

A straightforward computation yields
\begin{align}
\det(\overline A_U)
&=
-1+h(1-2\epsilon^2)+\epsilon^2+2\epsilon^4,
\label{eq:ex3-detAU}
\\
\det(\overline A_U(5))
&=
h(1-\epsilon^2)+\epsilon^4-2\epsilon^2.
\label{eq:ex3-detAU5}
\end{align}

Define
\(
F_U:\mathbb R\times\mathbb R\longrightarrow\mathbb R
\)
by
\(
F_U(h,\epsilon)
=
\det(\overline A_U(5)).
\)
Then
\(
F_U(0,0)=0
\)
and
\(
\left.
\frac{\partial F_U}{\partial h}
\right|_{(0,0)}
=1.
\)

Thus, the Implicit Function Theorem yields a unique continuously
differentiable function defined in a neighbourhood of \(0\) and
satisfying the required implicit equation.

From \eqref{eq:ex3-detAU5}, the equation
\(
F_U(h,\epsilon)=0
\)
is equivalent to
\[
h
=
\frac{2\epsilon^2-\epsilon^4}{1-\epsilon^2},
\]
provided that
\(
\epsilon\neq\pm1.
\)

Accordingly, define
\(
\psi:(-1,1)\longrightarrow\mathbb R
\)
by
\(
\psi(\epsilon)
=
\frac{2\epsilon^2-\epsilon^4}{1-\epsilon^2}.
\)
Then
\(
\psi(0)=0
\)
and
\(
F_U(\psi(\epsilon),\epsilon)=0
\)
for every \(\epsilon\in(-1,1)\). Hence \(\psi\) is the unique implicit
function obtained from the Implicit Function Theorem.

Substituting \(h=\psi(\epsilon)\) into \eqref{eq:ex3-detAU} gives
\[
\det A_U(\psi(\epsilon),\epsilon)
=
-\frac{(1-2\epsilon^2)^2}{1-\epsilon^2}.
\]

Therefore, the resulting matrix is non-singular whenever
\(
\epsilon\neq\pm\frac{1}{\sqrt2}
\)
and
\(
\epsilon\neq\pm1.
\)

Choosing
\(
\epsilon=\frac12,
\)
we obtain
\(
\psi\left(\frac12\right)=\frac7{12},
\)
and hence
\[
\mathcal A_U
=
\begin{bmatrix}
1&\frac12&0&0&0\\
\frac12&1&\frac12&0&\frac12\\
0&\frac12&\frac7{12}&\frac12&1\\
0&0&\frac12&1&0\\
0&\frac12&1&0&1
\end{bmatrix}.
\]

Thus, \(\mathcal A_W\) makes the vertices in \(W\) as \(P\)-vertices, while
\(\mathcal A_U\) makes the vertex in \(U\) as \(p\)-vertex. Therefore,
\(
p(G)\leq2.
\)

\end{example}
\begin{rem}
    This example illustrates both the constructions appearing in the proof of
Theorem~\ref{thm:main}. Unlike Example~\ref{ex:1}, the implicit
functions obtained here are nontrivial and the resulting matrices are
weighted matrices rather than the adjacency matrix of \(G\). Different
admissible choices of \(\epsilon\) produce different matrices realizing
the required \(P\)-vertex covering.
\end{rem}

\section*{Declaration of competing interest}
There is no competing interest.

\section*{Acknowledgements}
The first author's research was supported by the ANRF MATRICS Grant (File No. ANRF/ARGM/2025/002777/MTR) and by the Indian Institute of Technology Madras under Grant No. RG26271235MARGFX009003. The second author acknowledges financial support from the Indian Institute of Technology Madras for carrying out this research.

\end{document}